\documentclass[11pt]{amsart}

\usepackage{amsmath}
\usepackage{amsfonts}
\usepackage{amssymb}
\usepackage{array}
\usepackage{multirow}
\usepackage{mathtools}
\usepackage[english]{babel}
\usepackage{hyperref}

\newtheorem{thm}{Theorem}[section]
\newtheorem{lem}[thm]{Lemma}

\newtheorem{cor}[thm]{Corollary}

\theoremstyle{definition}
 
\newcommand{\Z}{{\mathbf{Z}}}
\newcommand{\Q}{{\mathbf{Q}}}
\newcommand{\OK}{{\mathcal K}}
\newcommand{\OO}{{\mathcal O}}

\newcommand{\floor}[1]{\left\lfloor#1\right\rfloor}
\newcommand{\p}{\mathfrak{p}}
\newcommand{\q}{\mathfrak{q}}
\newcommand{\hgt}{\textrm{h}}

\let\lm=\lambda
\let\Lm=\Lambda
\let\ep=\epsilon
\let\abs=\envert
\let\bs=\backslash

\newcommand{\acr}{\newline\indent}

\theoremstyle{remark}

\begin{document}
\title[An exponential diophantine equation]{An exponential diophantine equation related to odd perfect numbers}
\author[Tomohiro Yamada]{Tomohiro Yamada*}
\address{\llap{*\,}Center for Japanese language and culture\acr
                   Osaka University\acr
                   562-8558\acr
                   8-1-1, Aomatanihigashi, Minoo, Osaka\acr
                   JAPAN}
\email{tyamada1093@gmail.com}


\begin{abstract}
We shall show that, for any given primes $\ell\geq 17$ and $p, q\equiv 1\pmod{\ell}$,
the diophantine equation  $(x^\ell-1)/(x-1)=p^m q$ has at most four positive integral solutions $(x, m)$
and give its application to odd perfect number problem.
\end{abstract}

\maketitle

\section{Introduction}\label{intro}

The purpose of this paper is to bound the number of integral solutions of the diophantine equation
\begin{equation}\label{eq11}
\frac{x^\ell-1}{x-1}=p^m q, m\geq 0.
\end{equation}
This equation arises from our study of odd perfect numbers of a certain form.
$N$ is called perfect if the sum of divisors of $N$ except $N$ itself
is equal to $N$.
It is one of the oldest problem in mathematics whether or not an odd perfect number exists.
Euler has shown that an odd perfect number must be of the form
$N=p^{\alpha} q_1^{2\beta_1}\cdots q_k^{2\beta_k}$
for distinct odd primes $p, q_1, \ldots, q_k$ and positive integers
$\alpha, \beta_1, \ldots, \beta_r$ with $p\equiv \alpha\equiv 1\pmod{4}$.

However, we do not know a proof of the nonexistence of odd perfect numbers
even of the special form $N=p^{\alpha} (q_1 q_2 \cdots q_k)^{2\beta}$,
although McDaniel and Hagis conjecture that there exists no such one in \cite{MDH}.
Gathering various results such as \cite{CW}, \cite{FNO} \cite{HMD}, \cite{Ka1}, \cite{Mc}, \cite{MDH} and \cite{St},
we know that $\beta\geq 9$, $\beta\not\equiv 1\pmod{3}, \beta\not\equiv 2\pmod{5}$
and $\beta$ cannot take some other values such as $11, 14, 18, 24$.

We have shown that, if $N=p^{\alpha} (q_1 q_2 \cdots q_k)^{2\beta}$ is an odd perfect number,
then $k\leq 4\beta^2+2\beta+2$ in \cite{Ymd1}.
Recently, we have improved this upper bound by $2\beta^2+8\beta+2$
in \cite{Ymd3}, where the coefficient $8$ of $\beta$
can be replaced by $7$ if $2\beta+1$ is not a prime or $\beta\geq 29$.
Since it is known that $N<2^{4^{k+1}}$ from \cite{Nie}, we have
\[N<2^{4^{2\beta^2+8\beta+3}}.\]

The key point for this result is the diophantine lemma
that, if $\ell, p, q$ are given primes such that $\ell\geq 19$ and $p\equiv q\equiv 1\pmod{\ell}$, then
(\ref{eq11}) has at most six integral solutions $(x, m)$ such that $x$ is a prime below $2^{4^{\ell ^2}}$
and at most five such solutions if $\ell$ is a prime $\geq 59$
(we note that, by Theorems 94 and 95 in Nagell \cite{Nag},
any prime factor of $(x^\ell-1)/(x-1)$ with $\ell$ prime must be $\equiv 1\pmod{\ell}$ or equal to $\ell$).
Combining this result with an older upper bound in \cite{Ymd1}, we obtain the above upper bound
for $N$.

Now we return to the equation (\ref{eq11}),
which is a special type of Thue-Mahler equations.
Evertse gave an explicit upper bound for the numbers of solutions
of such equations.
Theorem 3 of \cite{Eve} gives that a slightly generalized equation
$(x^\ell-y^\ell)/(x-y)=p^m q^n$ has at most $2\times 7^{7(\ell-1)^3}$ integral solutions
for $\ell\geq 4$.  In this paper, we would like to obtain a stronger upper bound
for the numbers of solutions of (\ref{eq11}).

\begin{thm}\label{th1}
If $\ell, p, q$ are given primes such that $\ell\geq 19$ and $p\equiv q\equiv 1\pmod{\ell}$, then
(\ref{eq11}) has at most four positive integral solutions $(x, m)$.
Moreover, if (\ref{eq11}) has five integral solutions $(x_i, m_i)$ with $m_5>m_4>\cdots >m_1\geq 0$,
then $m_1=0$ and $x_2=x_1^r$ for some prime $r\neq\ell$.
\end{thm}

Combining this result with an argument in \cite{Ymd3}, we obtain the following new upper bound
for odd perfect numbers of a special form.
\begin{cor}
If $N=p^e (q_1 q_2 \ldots q_k)^{2\beta}$ is an odd perfect number
with $p, q_1, q_2$, $\ldots, q_k$ distinct primes and $p\equiv e\equiv 1\pmod{4}$,
then, $k\leq 2\beta^2+6\beta+2$ and $N<2^{4^{2\beta^2+6\beta+3}}$.
\end{cor}

Our method is similar to the approach used in \cite{Ymd3}.
In this paper, we use upper bounds for sizes of solutions of (\ref{eq11}) derived from
a Baker-type estimate for linear forms of logarithms by Matveev\cite{Mat},
which may be interesting itself, while \cite{Ymd3} used an older upper bound for odd perfect numbers
of the form given above.
We note that Pad\'{e} approximations using hypergeometric functions given by Beukers \cite{Beu1}\cite{Beu2}
does not work in our situation since our situation will give much weaker approximation to $\sqrt{D}$,
although Beukers' gap argument is still useful (see Lemma \ref{lm5} below).

In the next section, we introduce some arithmetic preliminary results from \cite{Ymd3} and Matveev's lower bound
for linear forms of logarithms.
In Section 3, using Matveev's lower bound, an upper bound
for the sizes of solutions of (\ref{eq11}) is given.
In Section 4, we prove Theorem \ref{th1}.
For large $\ell$, this can be done combining results in Sections 2 and 3
with general estimates for class numbers and regulators of quadratic fields.
For small $\ell$, we settle the case $x_1$ is large and then check the remaining $x_1$'s.

A more generalized equation of (\ref{eq11}) is
\begin{equation}\label{eq01}
\frac{x^\ell -1}{x-1}=y^m z^n, x\geq 2, y\geq 2, \ell\geq 3, mn\geq 2.
\end{equation}
Assuming the $abc$-conjecture, the author \cite{Ymd2} proved that
any integral solution of (\ref{eq01}) with $\ell\geq 3, m\geq 1, n\geq 2, 1\leq y<z$ and $x^\ell$ sufficiently large
must satisfy $(\ell, m, n)=(4, 1, 2), (3, 1, 3)$ or $(\ell, n)=(3, 2)$.

\section{A preliminary lemmas}
In this section, we shall introduce some notations and lemmas.

We begin by introducing a well-known result concerning prime factors
of values of the $n$-th cyclotomic polynomial, which we denote by $\Phi_n(X)$.
This result has been proved by Bang \cite{Ban} and rediscovered by many authors
such as Zsigmondy \cite{Zsi}, Dickson \cite{Dic} and Kanold \cite{Ka1, Ka2}.

\begin{lem}\label{lm2}
If $a$ is an integer greater than $1$, then $\Phi_n(a)$ has
a prime factor which does not divide $a^m-1$ for any $m<n$,
unless $(a, n)=(2, 6)$ or $n=2$ and $a+1$ is a power of $2$.
\end{lem}

In order to introduce further results on values of cyclotomic polynomials,
we need some notations and results from the arithmetic of a quadratic field.
Let $\ell\geq 17$ be a prime and $D=(-1)^\frac{\ell -1}{2} \ell$.
Let $\OK$ and $\OO$ denote $\Q(\sqrt{D})$ and its ring of integers $\Z[(1+\sqrt{D})/2]$ respectively.
We use the overline symbol to express the conjugate in $\OK$.
In the case $D>0$, $\ep$ and $R=\log\ep$ shall denote
the fundamental unit and the regulator in $\OK$ respectively.
In the case $D<-4$, we set $\ep=-1$ and $R=\pi i$.  We note that neither $D=-3$ nor $-4$ occurs
since we have assumed that $\ell\geq 17$.

Moreover, we define the absolute logarithmic height $\hgt(\alpha)$ of an algebraic number $\alpha$ in $\OK$.
For an algebraic number $\alpha$ in $\OK$ and a prime ideal $\p$ over $\OK$ such that
$\alpha=(\zeta_1/\zeta_2) \xi$ with $\xi\in \p^k$ and $\zeta_1, \zeta_2$ in $\OO\bs \p$,
we define the absolute value $\abs{\alpha}_\p$ by
\[
\abs{\alpha}_\p=N\p^{-k}
\]
as usual, where $N\p$ denotes the norm of $\p$, i.e., the rational prime lying over $\p$.
Now the absolute logarithmic height $\hgt(\alpha)$ is defined by
\[
\hgt(\alpha)=\frac{1}{2}\left(\log^+ \abs{\alpha}+\log^+ \abs{\bar\alpha}+\sum_\p \log^+\abs{\alpha}_\p\right),
\]
where $\log^+ t=\max\{0, \log t\}$ and $\p$ in the sum runs over all prime ideals over $\OK$.

The following three lemmas on the value of the cyclotomic polynomial $\Phi_\ell(x)$
are quoted from \cite{Ymd3}, except the latter part of Lemma \ref{lm4}.

\begin{lem}\label{lm3}
If $x$ is an integer $>\ell^2$, then $\Phi_\ell(x)$ can be written
in the form $X^2-DY^2$ for some coprime integers $X$ and $Y$
with $0.4387/x<\abs{Y/(X+Y\sqrt{D})}$ and $\abs{Y/(X-Y\sqrt{D})}<0.5608/x$.
Moreover, if $p, q$ are primes $\equiv 1\pmod{\ell}$ and $\Phi_\ell(x)=p^m q$ for some integer $m$, then,
\begin{equation}\label{eq21}
\left[\frac{X+Y\sqrt{D}}{X-Y\sqrt{D}}\right]=\left(\frac{\bar\p}{\p}\right)^{\pm m}\left(\frac{\bar\q}{\q}\right)^{\pm 1},
\end{equation}
where $[p]=\p\bar\p$ and $[q]=\q\bar\q$ are prime ideal factorizations in $\OO$.
\end{lem}

\begin{lem}\label{lm4}
Assume that $\ell$ is a prime $\geq 17$.
If $x_2>x_1>0$ are two multiplicatively independent integers and $\Phi_\ell(x_1)=p^{m_1} q$ and $\Phi_\ell(x_2)=p^{m_2} q$,
then $x_2>x_1^{\floor{(\ell +1)/6}}$.
If $x_2>x_1>0$ are multiplicatively dependent integers
and $\Phi_\ell(x_i)=p^{m_i} q$ for $i=1, 2$,
then $m_1=0$ and $x_2=x_1^r$ for some prime $r\neq\ell$.
\end{lem}

\begin{lem}\label{lm5}
If $\Phi_\ell(x_i)=p^{m_i} q$ for three integers $x_3>x_2>x_1>0$ with $x_2>x_1^{\floor{(\ell +1)/6}}$, then $m_3>0.445\abs{R}x_1/\sqrt{\ell}$.
\end{lem}

\begin{proof}[Proofs of lemmas]
The former statement of Lemma \ref{lm4} and Lemma \ref{lm5} are
4.1 and 4.2 of \cite{Ymd3} (the original version of Lemma 4.2 contains an error,
see the corrigendum) respectively for $\ell\geq 19$
and the corresponding stataments can be proved for $\ell=17$ in a similar way.
Moreover, Lemma \ref{lm3} is Lemma 2.3 of \cite{Ymd3} with $3^{\floor{(\ell+1)/6}}$
replaced by $\ell^2$ for $\ell\geq 19$ and can be proved in a similar way, even for $\ell=17$.
Hence, what we should prove here is only the latter statement of Lemma \ref{lm4}.

The assumption implies that $x_1=y^{r_1}$ and $x_2=y^{r_2}$ for some positive integers $y, r_1, r_2$
with $r_2>r_1$.
Assume that $r_1>1$ and put $r_i=s_i t_i$ with $t_i=\ell^{k_i}$ and $s_i$ not divisible by $\ell$
for $i=1, 2$.

If at least one $s_i\neq 1$,
then $\Phi_\ell(y^{r_i})$ must be divisible by $\Phi_{t_i \ell}(y)\Phi_{r_i \ell}(y)$.
Hence, three values $\Phi_{t_i \ell}(y)$, $\Phi_{r_1 \ell}(y)$ and $\Phi_{r_2 \ell}(y)$
must be composed only by $p$ and $q$.
However, since we have assumed that $\ell\geq 17$, Lemma \ref{lm2} yields that
each of $\Phi_{t_i \ell}(y)$, $\Phi_{r_1\ell}(y)$ and $\Phi_{r_2\ell}(y)$
must have a primitive prime factor.
This is a contradiction.

If $s_1=s_2=1$, then we have $t_1\neq t_2$ and $\Phi_\ell(x_i)=\Phi_{t_i \ell}(y)$ for $i=1, 2$.
Hence, both of $\Phi_{t_1 \ell}(y)$ and $\Phi_{t_2 \ell}(y)$ must be divisible by $q$,
which is impossible since $q\equiv 1\pmod{\ell}$.

Thus we must have $r_1=1$ and $x_2=x_1^r$.
If $r$ is divisible by $\ell$, then, writing $r=s\ell^k$ with $s$ indivisible by $\ell$,
we see that $p^{m_2} q=\Phi_\ell(x_1^r)=(x_1^{s\ell^{k+1}}-1)/(x_1^{s\ell^k}-1)=
\prod_{d\mid s}\Phi_{d\ell^{k+1}}(x_1)$.
If $s\neq 1$, then three values $\Phi_\ell(x_1), \Phi_{\ell^{k+1}}(x_1)$ and $\Phi_{s\ell^{k+1}}(x_1)$
must be composed only by $p$ and $q$, which is impossible like above.
Then $s=1$, and $q$ must divide both $\Phi_\ell(x_1)=p^{m_1} q$ and $\Phi_{\ell ^{k+1}}(x_1)$.
But this cannot occur since $q\equiv 1\pmod{\ell}$.

Hence, $r$ is not divisible by $\ell$ and
we see that $p^{m_2} q=(x_1^{r\ell}-1)/(x_1^r-1)=\prod_{d\mid r}\Phi_{d\ell}(x_1)$,
while each $\Phi_{d\ell}(x_1)$ has a primitive prime factor.
Hence, $r$ must be prime and, since $\Phi_\ell(x_1)$ must be divisible by $q$,
we conclude that $\Phi_{r\ell}(x_1)=p^{m_2}$ and $\Phi_{\ell}(x_1)=q$,
proving the latter statement of Lemma \ref{lm4}.
\end{proof}

In order to obtain an upper bound for the size of solutions, we use an lower bound for linear forms of logarithms due to Matveev\cite[Theorem 2.2]{Mat}.

\begin{lem}\label{lmll}
Let $\alpha_1, \alpha_2, \ldots, \alpha_n$ be algebraic integers in $\OO$
which are multiplicatively independent
and $b_1, b_2, \ldots, b_n$ be arbitrary integers.
Let \\ $A(\alpha)=\max\{2\hgt(\alpha), \abs{\log\alpha}\}$ and $A_j=A(\alpha_j)$.
Moreover, we put $\kappa=1$ if $D>0$ and $\kappa=2$ if $D<0$. 

Put
\begin{equation}
\begin{split}
B= & \max \{1, \abs{b_1}A_1/A_n, \abs{b_2}A_2/A_n, \ldots, \abs{b_n} \},\\
\Omega= & A_1A_2\ldots A_n,\\
C_\kappa(n)= & \frac{16}{n!\kappa}e^n(2n+1+2\kappa)(n+2)(4(n+1))^{n+1} \\
& \times \left(\frac{1}{2}en\right)^\kappa(4.4n+5.5\log n+7+2\log 2+\log(1+\log(2))), \\
c & =3e(1+\log 2)
\end{split}
\end{equation}
and
\begin{equation}
\Lm=b_1\log \alpha_1+\ldots+b_n\log \alpha_n.
\end{equation}
Then we have $\Lm=0$ or
\begin{equation}
\log\abs{\Lm}>-C_\kappa(n)(\log cB)\max\left\{1, \frac{n}{6}\right\}\Omega.
\end{equation}
\end{lem}

\section{Upper bounds for the sizes of solutions}

In this section, we shall give upper bounds for the sizes of solutions of (\ref{eq11}),
which itself may be of interest.
As in the previous sections, for a prime $\ell\geq 17$, we let
$D=(-1)^\frac{\ell -1}{2} \ell$, $\OK$ and $\OO$ denote the quadratic field $\Q(\sqrt{D})$
and its ring of integers $\Z[(1+\sqrt{D})/2]$ respectively
and $h$ be the class number of $\OK$.
In the case $D>0$, $\ep$ and $R=\log\ep$ shall denote
the fundamental unit and the regulator in $\OK$ respectively.
In the case $D<-4$, we set $\ep=-1$ and $R=\pi i$.
We note that $\abs{R}>\log(\sqrt{17})>1.4$ for every $D$ with $\abs{D}\geq 17$.

We let $p, q$ be primes $\equiv 1\pmod{\ell}$.
Then we can factor $[p]=\p\bar\p$ and $[q]=\q\bar\q$ in $\OO$
and we see that $\p^h=[\tau]$ and $\q^h=[\eta]$ for some $\tau, \eta\in \OO$.
In the case $D>0$, taking integers $u, v$ so that
$\abs{\tau\ep^u}\leq p^{h/2}\ep^{1/2}\leq \abs{\tau\ep^{u+1}}$
and $\abs{\eta\ep^v}\leq p^{h/2}\ep^{1/2}\leq \abs{\eta\ep^{v+1}}$
we can take $\tau_0=\tau\ep^u$ and $\eta_0=\eta\ep^v$ in $\OO$ such that $[\tau_0]=\p^h, [\eta_0]=\q^h$
and $p^{h/2} \ep^{-1/2}\leq \abs{\tau_0}\leq p^{h/2} \ep^{1/2}, q^{h/2} \ep^{-1/2}\leq \abs{\eta_0}\leq q^{h/2} \ep^{1/2}$.
In the case $D<0$, we can easily observe that
$\tau_0$ and $\eta_0$ in $\OO$ can be chosen from
$\pm\tau, \pm\bar\tau$ and $\pm\eta, \pm\bar\eta$ respectively
such that $[\tau_0]=\p^h, [\eta_0]=\q^h$ and $\abs{\arg \tau_0}, \abs{\arg \eta_0}<\pi/4$.

\begin{thm}\label{th2}
Assume that $\Phi_\ell(x)=p^m q$ and put $C=C_1(3)=1.813\cdots \times 10^{10}\cdots$ if $\ell\equiv 1\pmod{4}$
and $C=C_2(3)=4.518\cdots \times 10^{10}$ if $\ell\equiv 3\pmod{4}$.
Then we have the following upper bounds for $m$:
\begin{itemize}
\item[i)] If $h\log q>h\log p\geq \abs{R}$, then 
\begin{equation}\label{eq301}
m<4.505C\ell h^2 \abs{R}(\log q)(\log (8cC\ell h^2 \abs{R})+\log \log p).
\end{equation}
\item[ii)] If $h\log q\geq \abs{R}\geq h\log p$, then
\begin{equation}
m<4.505C\frac{\ell}{\log (2\ell)} h \abs{R}^2 (\log q)\log \left(\frac{8cC\ell \abs{R}^3}{\log(2\ell)}\right).
\end{equation}
\item[iii)] If $h\log p>h\log q\geq \abs{R}$, then
\begin{equation}
m<4.505C\ell h^2 \abs{R}(\log q)(\log (8cC\ell h^2 \abs{R})+\log\log q).
\end{equation}
\item[iv)] If $h\log p\geq \abs{R}\geq h\log q$, then
\begin{equation}
m<4.505C\ell h \abs{R}^2\log (8cC\ell h \abs{R}^2).
\end{equation}
\item[v)] If $\abs{R}\geq h\log \max\{p, q\}$, then
\begin{equation}\label{eq302}
m<4.505C\ell \abs{R}^3\frac{\log (8cC\ell \abs{R}^3)}{\log\ell}.
\end{equation}
\end{itemize}
\end{thm}

\begin{proof}
We begin by observing that if $m\leq 2\ell\log\ell$,
then we can easily confirm the Theorem exploiting the fact that $p, q>2\ell$.
Indeed, in cases i), iii) and iv), we have $2\ell\log\ell<C\ell \log \ell$ is clearly smaller than
the right hand side of the desired inequality in each case.
Moreover, in cases ii) and v), we have $\abs{R}\geq \log p>\log (2\ell)$
and $2\ell\log\ell<C\ell \log\abs{R}$ is smaller than the right hand side
of the desired inequality in each case.
Hence, we may assume that $m>2\ell\log \ell$, so that $x>\ell^2$.
If $\Phi_\ell(x)=p^m q$, then Lemma \ref{lm3} yields that
there exist two integers $X, Y$ such that
\begin{equation}
\left[\frac{X+Y\sqrt{D}}{X-Y\sqrt{D}}\right]=\left(\frac{\bar\p}{\p}\right)^{\pm m}\left(\frac{\bar\q_j}{\q_j}\right)^{\pm 1},
\end{equation}
with $0<\abs{Y/(X-Y\sqrt{D})}<0.5608/x$.
We can easily see that $(X+Y\sqrt{D})/(X-Y\sqrt{D})\neq \pm 1$ from $Y/(X-Y\sqrt{D})\neq 0$.
Since $\abs{D}=\ell>3$ is odd, $(X+Y\sqrt{D})/(X-Y\sqrt{D})$ cannot be a root of unity.
Hence, taking the $h$-th powers, we have
\begin{equation}\label{eq31}
\left(\frac{X+Y\sqrt{D}}{X-Y\sqrt{D}}\right)^h=\ep^u \left(\frac{\bar{\tau_0}}{\tau_0}\right)^{\pm m}\left(\frac{\bar{\eta_0}}{\eta_0}\right)^{\pm 1}\neq 1
\end{equation}
for some integer $u$,
where we take $\tau_0$ and $\eta_0$ as we explained just before the lemma.
Now let
\begin{equation}
\Lm=u\log \ep\pm m\log \left(\frac{\bar{\tau_0}}{\tau_0}\right)\pm \log\left(\frac{\bar{\eta_0}}{\eta_0}\right)
=h\log \left(\frac{X+Y\sqrt{D}}{X-Y\sqrt{D}}\right).
\end{equation}
Then, proceeding as in the corrigendum of \cite{Ymd3}, (\ref{eq31}) gives that
\begin{equation}
0<\abs{\Lm}<\frac{2hY\sqrt{\ell}}{\abs{X-Y\sqrt{D}}}<\frac{1.1216h\sqrt{\ell}}{x}.
\end{equation}

If $\abs{\Lm}\geq 1$, then we have $x<1.1216h\sqrt{\ell}<h\ell$ and $m<\ell\log x/\log p<\ell\log (h\ell)$.
We can easily confirm the desired inequality in each case.
Hence, we may assume that $\abs{\Lm}<1$.

Before applying Lemma \ref{lmll}, we must obtain upper bounds for $A(\ep)$, $A(\bar\tau_0/\tau_0)$
and $A(\bar\eta_0/\eta_0)$.
If $D>0$, then we deduce from $p^{h/2}\ep^{-1/2}\leq \abs{\tau_0}\leq p^{h/2}\ep^{1/2}$ that
$\abs{\bar\tau_0/\tau_0}\leq \ep$ and $h\log p\leq 2\hgt(\bar\tau_0/\tau_0)\leq h\log p+\log\abs{\ep}$.
Thus, we obtain $h\log p\leq A(\bar\tau_0/\tau_0)\leq h\log p+\abs{R}$
and, similarly, we obtain $h\log q\leq A(\bar\eta_0/\eta_0)\leq h\log q+\abs{R}$.
Moreover, since $\hgt(\ep)=(\log\ep)/2$, we have $A(\ep)\leq \abs{R}$.
If $D<0$, then the situation become simpler.
We can see that $\abs{\tau_0}=\abs{\bar\tau_0}=\abs{\eta_0}=\abs{\bar\eta_0}=p^{h/2}$.
Hence, $\hgt(\bar\tau_0/\tau_0)=\log\abs{\tau_0}_\p^{-1}=(h/2)\log p$ and similarly
$\hgt(\bar\eta_0/\eta_0)=(\log q)/2$.
Now we have $A(\bar\tau_0/\tau_0)=\max\{h\log p, \pi/2\}=h\log p$ since $p\geq 47>e^{\pi/2}$
and similarly $A(\bar\eta_0/\eta_0)=h\log q$.
Moreover, $A(\ep)=A(-1)=\pi=\abs{R}$.
Thus, in any case, we obtain
$h\log p\leq A(\bar\tau_0/\tau_0)\leq h\log p+\abs{R}$,
$h\log q\leq A(\bar\eta_0/\eta_0)\leq h\log q+\abs{R}$
and $A(\ep)\leq \abs{R}$.

We begin by treating the first case $h\log q>h\log p>\abs{R}$.
We have
\begin{equation}\label{eq32}
\frac{mA(\bar\tau_0/\tau_0)}{A(\bar\eta_0/\eta_0)}=\frac{m(h\log p+\abs{\log(\bar\tau_0/\tau_0)}}{h\log q+\abs{\log(\bar\eta_0/\eta_0)}}\leq \frac{m(h\log p+\abs{R})}{h\log q}
\end{equation}
and
\begin{equation}\label{eq33}
\begin{split}
\frac{uA(\ep)}{A(\bar\eta_0/\eta_0)} & =\frac{\abs{u\log\ep}}{A(\bar\eta_0/\eta_0)} \\
& \leq \frac{m\abs{\log(\bar\tau_0/\tau_0)}+\abs{\log(\bar\eta_0/\eta_0)}+\abs{\Lm}}{h\log q} \\
& <\frac{(m+1)\abs{R}+\abs{\Lm}}{h\log q} \\
& <\frac{2m\abs{R}}{h\log q},
\end{split}
\end{equation}
where we recall that $\abs{\Lm}<1$ and
observe that $m>2\ell\log\ell>48, \abs{R}>1.4$ and $(m+1)\abs{R}+\abs{\Lm}<(m+1)\abs{R}+1<2m\abs{R}$.

Since $h\log q>h\log p>\abs{R}$, we see that
$A(\bar\tau_0/\tau_0)<h\log p+\abs{R}<2h\log p$, $A(\bar\eta_0/\eta_0)<h\log q+\abs{R}<2h\log q$
and $B\leq 2m\log p/\log q$.
Hence, Matveev's theorem gives
\begin{equation}
\begin{split}
& \log x-\log(1.1216h\sqrt{\ell})<-\log\abs{\Lm} \\
< & C(2h)^2 \log\left(\frac{2cm\log p}{\log q}\right)\abs{R}(\log p)(\log q)
\end{split}
\end{equation}
and therefore
\begin{equation}
\begin{split}
& \frac{m\log p}{\log q}<\frac{\ell\log x}{\log q} \\
< & \ell\left(\frac{\log(1.1216h\sqrt{\ell})}{\log q}
+4Ch^2 \abs{R}\log\left(\frac{2m\log p}{\log q}\right)(\log p)\right).
\end{split}
\end{equation}
Taking it into account that $C>10^{10}$, we may assume that \\ $(2cm\log p)/\log q>10^{10}$;
otherwise \eqref{eq301} automatically holds.
Now we observe that $q, p\geq \max\{\ell, 47\}$ and
$2c\log(1.1216h\sqrt{\ell})/\log q<1+\log h+\log\ell\leq 1+\log h+\log p<2h\log p$.
Hence, we obtain
\begin{equation}\label{eq34}
\begin{split}
\frac{2cm\log p}{\log q} & <(8cC+1)\ell h^2 \abs{R}\log\left(\frac{2cm\log p}{\log q}\right)(\log p) \\
& =:U\log\left(\frac{2cm\log p}{\log q}\right).
\end{split}
\end{equation}

In other words, putting $W=2cm\log p/\log q$, we have $W/\log W<U$.
Since $U>8cC\geq 8cC_1(3)>2\times 10^{12}$, we have $(\lm U\log U)/\log(\lm U\log U)<U$
with $\lm=1.12212$.
Thus we obtain $W<\lm U\log U$.
Noting that $8cC>2\times 10^{12}$, $\ell\geq 17$, $\abs{R}>1.4$ and $p\geq 47$,
we have $\log(8cC\ell\abs{R}\log p)>32.84$ and
\begin{equation}
\begin{split}
\log U= & \log (\lm (8cC+1)\ell h^2\abs{R}\log p) \\
< & \log(8cC\ell h^2\abs{R}\log p)+\log(1.12213) \\
< & 1.00351\log (8cC\ell h^2\abs{R}\log p).
\end{split}
\end{equation}
Hence, (\ref{eq34}) yields that
\begin{equation}
\begin{split}
& \frac{2cm\log p}{\log q}< \lm U\log U \\
& \quad <\lm(8cC+1)\ell h^2\abs{R}(\log p)\times 1.00351\log (8cC\ell h^2\abs{R}(\log p))
\end{split}
\end{equation}
and, dividing by $2c$,
\begin{equation}
\frac{m\log p}{\log q}<4.505C\ell h^2 \abs{R}(\log p)(\log (\ell h^2 \abs{R})+\log \log p+\log (8cC)),
\end{equation}
proving i).

Then, if $h\log q>\abs{R}>h\log p$, then $A(\bar\tau/\tau)<2\abs{R}, A(\bar\eta/\eta)<2h\log q$
and $B\leq 2m\abs{R}/h\log q$.
Moreover, (\ref{eq32}) and (\ref{eq33}) hold as in the previous case.
Hence, an argument similar to above yields that
\begin{equation}
\frac{m\log p}{\log q}<\ell\left(\frac{\log(1.1216h)}{\log q}+4Ch\abs{R}^2\log\left(\frac{2cm\abs{R}}{h\log q}\right)\right)
\end{equation}
and, observing that $p>2\ell$,
\begin{equation}
\frac{m\abs{R}}{h\log q}<\ell\left(\frac{\abs{R}\log(1.1216h)}{h(\log (2\ell))(\log q)}+4C\frac{\abs{R}^3}{\log (2\ell)}\log\left(\frac{2cm\abs{R}}{h\log q}\right)\right).
\end{equation}

We see that $\log(8cC\ell\abs{R}^3/\log(2\ell))>\log(8cC\ell(\log^3 p)/\log(2\ell))>33.85$ in this case.
Thus, proceeding as above, we obtain
\begin{equation}
\frac{m\abs{R}}{h\log q}<4.505C\frac{\ell}{\log (2\ell)} \abs{R}^3 \log\left(\frac{8cC\ell \abs{R}^3}{\log(2\ell)}\right),
\end{equation}
which proves ii).

In the remaining cases, similar arguments give iii), iv) and v).
\end{proof}

\section{Proof of the main theorem}

In this section, we shall prove the main theorem.

Assume that $\Phi_\ell(x_i)=p^{m_i} q$ has five solutions $x_1<x_2<x_3<x_4<x_5$
such that $x_1$ and $x_2$ are multiplicatively independent.
It is clear that $x_1\geq \max\{q^{1/\ell}, 2\}$.
Since we have assumed that $x_1$ and $x_2$ are multiplicatively independent,
Lemma \ref{lm4} yields that $x_3\geq\max\{q, 2^\ell\}^{\floor{(\ell +1)/6}^2/\ell}$.
Now it follows from Lemma \ref{lm5} that
\begin{equation}\label{eq41}
m_5>0.397\pi x_3>0.397\pi \max\{q^{\floor{(\ell +1)/6}^2/\ell}, 2^{\floor{(\ell +1)/6}^2}\}:=M.
\end{equation}

We begin by the case $\ell\geq 47$.
If $\ell\equiv 3\pmod{4}$, then $R=\pi i$.
If $\ell\equiv 1\pmod{4}$, then, noting that $\ell$ is prime,
it follows from Proposition 3.4.5 of \cite[p. 138]{HC1} and Proposition 10.3.16 of \cite[p. 200]{HC2}
that $hR<\ell ^{1/2}((\log\ell)/2+\log\log\ell+2.8)$
($hR\leq \ell ^{1/2}\log (4\ell)$ in p.199 of \cite{Fai} can also be used).
Now Theorem \ref{th2} implies that $m_5<M$, which contradicts to (\ref{eq41}).
Hence, if $\ell\geq 47$, then $\Phi_\ell(x)=p^m q$ can never have five solutions
$x_1<\cdots <x_5$ such that $x_1$ and $x_2$ are pairwise multiplicatively independent.

Next, assume that $\ell=43$.
We must have $x_1\geq 3$ since $2^{43}-1=431\times 9719\times 2099863$ has
three distinct prime factors.
Thus we must have $m_5>0.397\pi \max\{q^{49/43}, 3^{49}\}$,
which exceeds the upper bounds given in Theorem \ref{th2} with $h=1, R=\pi i$.
Indeed, Theorem \ref{th2} would yield that,
if $q<3^{43}$, then $m_5<5\times 10^{16}<0.397\pi\times 3^{49}<m_5$ and,
if $q>3^{43}$, then $m_5<2.8\times 10^{13}(\log q)(\log\log q+35)<0.397\pi q^{49/43}<m_5$.
In both cases, we are led to a contradiction.
Hence, $\Phi_{43}(x_i)=p^{m_i} q$ can never have five solutions $x_1<\cdots <x_5$
such that $x_1$ and $x_2$ are pairwise multiplicatively independent.

\begin{table}
\caption{Estimates when $\ell\leq 41$ and $x_1$ is large}
\begin{center}
\begin{small}
\begin{tabular}{| c | c | c | c | c | c |}
 \hline
$\ell$ & $h$ & $R$ & $x_1\geq$ & $x_2>$ & $x_3>$ \\
 \hline
$17$ & $1$ & $\log (4+\sqrt{17})$ & $60$ & $x_1^3$ & $\max\{q^{9/17}, 60^9\}$ \\
$19$ & $1$ & $\pi i$ & $68$ & $x_1^3$ & $\max\{q^{9/19}, 68^9\}$ \\
$23$ & $3$ & $\pi i$ & $13$ & $x_1^4$ & $\max\{q^{16/23}, 13^{16}\}$ \\
$29$ & $1$ & $\log ((5+\sqrt{29})/2)$ & $5$ & $x_1^5$ & $\max\{q^{25/29}, 5^{25}\}$ \\
$31$ & $3$ & $\pi i$ & $5$ & $x_1^5$ & $\max\{q^{25/31}, 5^{25}\}$ \\
$37$ & $1$ & $\log (6+\sqrt{37})$ & $3$ & $x_1^6$ & $\max\{q^{36/37}, 3^{36}\}$ \\
$41$ & $1$ & $\log (32+5\sqrt{41})$ & $3$ & $x_1^7$ & $\max\{q^{49/41}, 3^{49}\}$ \\
 \hline
\end{tabular}
\label{tbl1}
\end{small}
\end{center}
\end{table}

\begin{table}
\caption{Estimates when $\ell\leq 41, m_1>0$ and $x_1$ is small}
\begin{center}
\begin{small}
\begin{tabular}{| c | c | c | c | c | c |}
 \hline
$\ell$ & $x_1$ & $p, q\geq$ & $p, q\leq$ \\
 \hline
\multirow{3}{*}{$17$} & $3, 4, 5, 7, 10, 12, 14, 15$, & \multirow{3}{*}{$103$} & \multirow{3}{*}{$362759437743508955104646759$} \\
& $19, 23, 26, 32, 39, 41, 42$, & & \\
& $44, 45, 46, 48, 58, 61$ & & \\
 \hline
\multirow{3}{*}{$19$} & $3, 4, 6, 7, 13, 15, 18, 21$, & \multirow{3}{*}{$191$} & \multirow{3}{*}{$607127818287731321660577427051$} \\
& $26, 28, 29, 30, 33, 34, 35$, & & \\
& $37, 38, 50, 61, 62, 63$ & & \\
 \hline
$23$ & $2, 3, 5$ & $47$ & $332207361361$ \\
$37$ & $2$ & $223$ & $616318177$ \\ 
$41$ & $2$ & $13367$ & $164511353$ \\
 \hline
\end{tabular}
\label{tbl2}
\end{small}
\end{center}
\end{table}

If $\ell\leq 41$ and $x_1$ is not less than the corresponding value given in Table \ref{tbl1},
then $x_2$ and $x_3$ exceeds the value given in this table.
Now we see that $m_5>0.397\pi x_3$ exceeds our upper bound $M$, which leads to a contradiction.

Now we shall examine the remaining cases.
Then $m_1=0$ or $x_1$ must be one of the values given in Table \ref{tbl2} and
$p, q$ must be in the range given in this table.

Assume that $x_1$ is one of the values given in Table \ref{tbl2}.
In any case, Theorem \ref{th2} gives that $m<1.37\times 10^{17}$.
But we have confirmed that $x_2>p^4\geq 47^4>10^6$ for these cases.
Hence, we must have $x_3>x_2^4>10^{24}$ and $m_5>x_3>10^{24}$
for all cases given in Table \ref{tbl2}, which is a contradiction again.

For example, in the case $\ell=23$ (in this case, we have $h=3$ and $R=\pi i$),
if $x_1\geq 13$, then we must have $m_5>0.397\pi \max\{q^{16/23}, 13^{16}\}$,
which exceeds the upper bounds given in Theorem \ref{th2}.

If $x_1<13$, then we must have $x_1=2, 3, 5$;
$(10^{23}-1)/9$ is prime and $(x^{23}-1)/(x-1)$ with $x=4, 6, 7, 8, 9, 11$ or $12$
has more than two distinct prime factors.

If $x_1=2, 3$ or $5$, then $p, q\leq 332207361361$ and $m<1.3\times 10^{17}$.
But, in any case, we have confirmed that $x_2>p^4>10^6$.
Hence, we must have $x_3>x_2^4>10^{24}$ and $m_5>x_3>10^{24}$,
which is a contradiction.

Next assume that $m_1=0$ or, equivalently, $(x_1^\ell -1)/(x_1-1)=q$.
Thus, $(x, \ell)=(2, 31), (10, 23)$ or
$\ell=19, x\in \{2, 10, 11, 12, 14, 19, 24, 40, 45, 46, 48, \\
65, 66, 67\}$
or $\ell=17, x\in \{2, 11, 20, 21, 28, 31, 55, 57\}$.
We observe that $x_1^r$ with $1\leq r\leq \ell -1$ give the complete solutions
to the congruence $(X^\ell -1)/(X-1)\equiv 0\pmod{q}$ and $x_1^{\ell -1}<q$.
Since $x_1, x_2$ are multiplicatively independent, we must have $x_2>x_1^\ell >\max\{q, 2^\ell\}$.
Thus, $x_3>x_2^{\floor{(\ell +1)/6}}$ and $m_5>0.397\abs{R} x_3$.
However, in any case, this exceeds the upper bound for $m$ given by Theorem \ref{th2}.

For example, if $\ell=17$ and $x_1=2$, then $x_2\geq 2^{17}$,
$x_3>x_2^3$ and $m_5>0.397\log(4+\sqrt{17}) x_3>1.8\times 10^{15}$
while Theorem \ref{th2} gives $m_5<1.2\times 10^{15}$.

Thus, we have proved that if $\Phi_\ell(x_i)=p^{m_i} q$ has five solutions
$x_1<x_2<x_3<x_4<x_5$, then $x_1$ and $x_2$ are multiplicatively dependent.
Combining it with Lemma \ref{lm4}, the proof of Theorem \ref{th1} is completed.

{}
\end{document}